\documentclass[a4paper,10pt]{article}
\usepackage{amsfonts,amssymb,amsmath,graphicx}
\usepackage{theorem,times}
\usepackage[english]{babel}
{\theorembodyfont{\rm}

}

  \newtheorem{thm}{Theorem}%[defi]

\newcommand{\para}{\vspace{3pt plus 2pt}
\refstepcounter{defi}\textbf{\arabic{section}.\arabic{defi}.}\, }

\newcommand{\vo}{\mbox{{\boldmath$o$}}}
\newcommand{\vu}{\mbox{{\boldmath$u$}}}
\newcommand{\vx}{\mbox{{\boldmath$x$}}}

\newcommand{\PP}{{\mathbb P}}
\newcommand{\RR}{{\mathbb R}}
\newcommand{\T}{{\mathrm{T}}}
\newcommand{\dis}%{{\mathrel{\vartriangle}}}
                 {{\mathrel{\scriptstyle{\triangle}}}}

\newcommand{\GL}{{\mathrm{GL}}}
\newcommand{\diag}{{\mathrm{diag}}}

\newcommand{\abc}{{\alpha,\beta,\gamma}}
\newcommand{\alphaq}{{\overline{\vphantom{\beta}\alpha}}}
\newcommand{\betaq}{{\overline\beta}}
\newcommand{\gammaq}{{\overline{\vphantom{\beta}\gamma}}}
\newcommand{\abcq}{{\,\alphaq,\betaq,\gammaq\,}}

\let\phi=\varphi

\let\theta=\vartheta

\newcommand{\weit}{\renewcommand{\arraystretch}{1.15}}
\weit
\newcommand{\schieb}[1]{\hspace{#1mm}}
\newcommand{\fuell}[1]{\hfill{#1}\hfill}
\newcommand{\abst}[1]{\hspace{0mm}{#1}\hspace{0mm}}

\newcommand{\Matrixfeld}[4]{\left#1\!\begin{array}{*{#3}{c}}#4\end{array}\!\right#2}

\newcommand{\RMatrixfeld}[4]{\left#1\!\begin{array}{*{#3}{r}}#4\end{array}\!\right#2}

\newcommand{\Mat}{\Matrixfeld()}

\newcommand{\RMat}{\RMatrixfeld()}

\newenvironment{proof}%[1]%
    {\begin{trivlist} \item {\emph{Proof}.}} %#1\/:}}%
    {{}\hfill $\square$ \end{trivlist}} %%Hans {}\hfill erg{\"a}nzt und \/ gel{\"o}scht

\newcounter{abbildung} %% Zaehler fuer Abbildungen

\date{}
\begin{document}
\title{Cayley's surface revisited}
\author{Hans Havlicek}

\maketitle

\centerline{\emph{Dedicated to Gunter Wei{\ss} on the occasion of his 60{th}
birthday, in friendship}}

%%%%%%%%%%%%%%%%%%%%%%%%%%%%%%%
\begin{abstract}
Cayley's (ruled cubic) surface carries a three-parameter family of
twisted cubics. We describe the contact of higher order and the
dual contact of higher order for these curves and show that there
are three exceptional cases.
\end{abstract}

\section{Introduction}\label{se:introduction}

\para The geometry on Cayley's surface and the geometry in the
ambient space of Cayley's surface has been investigated by many
authors from various points of view. See, among others,
\cite{brau-67c}, \cite{husty-84}, \cite{koch-68},
\cite{oehler-69}, and \cite{wiman-36}. In these papers the reader
will also find a lot of further references.
\par
As a by-product of a recent publication \cite{havl-list-03z}, it turned out
that the Cayley surface (in the real projective $3$-space) carries a
one-parameter family of twisted cubics which have mutually contact of order
four. These curves belong to a well-known three-parameter family of twisted
cubics $c_\abc$ on Cayley's surface; cf.\ formula (\ref{eq:abc_proj}) below.
All of them share a common point $U$ with a common tangent $t$, and a common
osculating plane $\omega$, say. However, according to
\cite[pp.~96--97]{brau-64} such a one-parameter family of twisted cubics with
contact of order four should not exist: ``\emph{Zwei Kubiken dieser Art, die
einander in $U$ mindestens f\"unfpunktig ber\"uhren, sind identisch.''}
\par
The aim of the present communication is to give a complete description of the
order of contact (at $U$) for the twisted cubics mentioned above. In
particular, it will be shown in Theorem \ref{thm:1} that the twisted cubics
with parameter $\beta=\frac32$ play a distinguished role, a result that seems
to be missing in the literature. Furthermore, since the order of contact is not
a self-dual notion, we also investigate the order of dual contact for twisted
cubics $c_\abc$. Somewhat surprisingly, in the dual setting the parameters
$\beta=\frac52$ and $\beta=\frac73$ are exceptional; see Theorem \ref{thm:2}.
\par
In Section \ref{para:2.4} we show that certain results of Theorem
\ref{thm:1} have a natural interpretation in terms of the
\emph{twofold isotropic geometry} which is based on the absolute
flag $(U,t,\omega)$, and in terms of the \emph{isotropic geometry}
in the plane $\omega$ which is given by the flag $(U,t)$. Section
\ref{para:3.2} is devoted to the interplay between Theorem
\ref{thm:1} and Theorem \ref{thm:2}.
\par\para
The calculations which are presented in this paper are long but
straightforward. Hence a computer algebra system (Maple V) was
used in order to accomplish this otherwise tedious job.
Nevertheless, we tried to write down all major steps of the
calculations in such a form that the reader may verify them
without using a computer.
\par

\section{Contact of higher order}\label{se:contact}

\para
Throughout this paper we consider the three-dimensional real
projective space $\PP_3(\RR)$. Hence a point is of the form
$\RR\vx$ with $\vx=(x_0,x_1,x_2,x_3)^\T$ being a non-zero vector
in $\RR^{4\times 1}$. We choose the plane $\omega$ with equation
$x_0=0$ as \emph{plane at infinity}, and we regard $\PP_3(\RR)$ as
a projectively closed affine space. For the basic concepts of
projective differential geometry we refer to \cite{bol-50} and
\cite{degen-94}.

\par\para
The following is taken from \cite{brau-64}, although our notation
will be slightly different. \emph{Cayley's} (\emph{ruled cubic\/})
\emph{surface} is, to within collineations of $\PP_3(\RR)$, the
surface $F$ with equation
\begin{equation}\label{eq:cayley}
  3x_0x_1x_2-x_1^3-3x_3x_0^2=0.
\end{equation}
The line $t:x_0=x_1=0$ is on $F$. More precisely, it is a torsal generator of
second order and a directrix for all other generators of $F$. The point
$U=\RR(0,0,0,1)^\T$ is the cuspidal point on $t$. In Figure \ref{abb1} a part
of the surface $F$ is displayed in an affine neighbourhood of the point $U$. In
contrast to our general setting, $x_3=0$ plays the role of the plane at
infinity in this illustration.

\par
On the surface $F$ there is a three-parameter family of cubic
parabolas which can be described as follows: Each triple
$(\abc)\in\RR^3$ with $\beta\neq 0$ gives rise to a function
\begin{equation*}%%LAYOUT
\begin{array}{l}
\Phi_\abc :\RR^{2\times 1}\to \RR^{4\times 1}: \vu=(u_0,u_1)^\T\mapsto\hfill \\
\schieb{2.0} \displaystyle
  \left(u_0^3, u_0^2(u_1-\gamma u_0),
 \frac{u_0(u_1^2+\alpha u_0^2)}{\beta},
 \frac{(u_1-\gamma u_0)}{3\beta}\left(3(u_1^2+\alpha u_0^2) - \beta(u_1-\gamma
 u_0)^2\right)\right)^\T\schieb{-0.5}.
 \end{array}
\end{equation*}
If moreover $\beta\neq 3$ then $\Phi_\abc$ yields the mapping
\begin{equation}\label{eq:abc_proj}
   \PP_1(\RR)\to \PP_3(\RR): \RR\vu\mapsto \RR(\Phi_\abc(\vu));
\end{equation}
its image is a \emph{cubic parabola\/} $c_\abc\subset F$. All
these cubic parabolas have the common point $U$, the common
tangent $t$ and the common osculating plane $\omega$. We add in
passing that for $\beta=3$ we have
$\Phi_{\alpha,3,\gamma}\left((0,u_1)^\T\right)=\vo$ for all
$u_1\in\RR$, whereas the points of the form
$\RR\left(\Phi_{\alpha,3,\gamma}((1,u_1)^\T)\right)$ comprise the
affine part of a \emph{parabola\/}, $c_{\alpha,3,\gamma}$ say,
lying on $F$. Each curve $c_\abc$ ($\beta\neq 0$) is on the
\emph{parabolic cylinder\/} with equation
\begin{equation}\label{eq:zylinder}
    \alpha x_0^2-\beta x_0x_2+(x_1+\gamma x_0)^2 = 0.
\end{equation}
The mapping $( \abc ) \mapsto c_\abc$ is injective, since
different triples $(\abc)$ yield different parabolic cylinders
(\ref{eq:zylinder}).
\par
Figure \ref{abb2} shows some generators of $F$, and five cubic
parabolas $c_{\alpha,\beta,0}$ together with their corresponding
parbolic cylinders, where $\alpha$ ranges in
$\{-\frac32,-\frac34,0,\frac34,\frac32 \}$ and $\beta=\frac32$.
%%%%%%%%%%%%%%%%%%%%%%%%%%%%%%%%%%%%%%%%%%%%%%%%%%%%%%%%%%%%%%%%%%
{\unitlength1.0cm
      \begin{center}
      %% Figur 2000 2491 = 4.0 x 4.98 feine Aufloesung 1270dpi
      \begin{minipage}[t]{4.0cm}
         \begin{picture}(4.0,5.18)
         \put(0.0 ,0.0){\includegraphics[width=4.0cm]{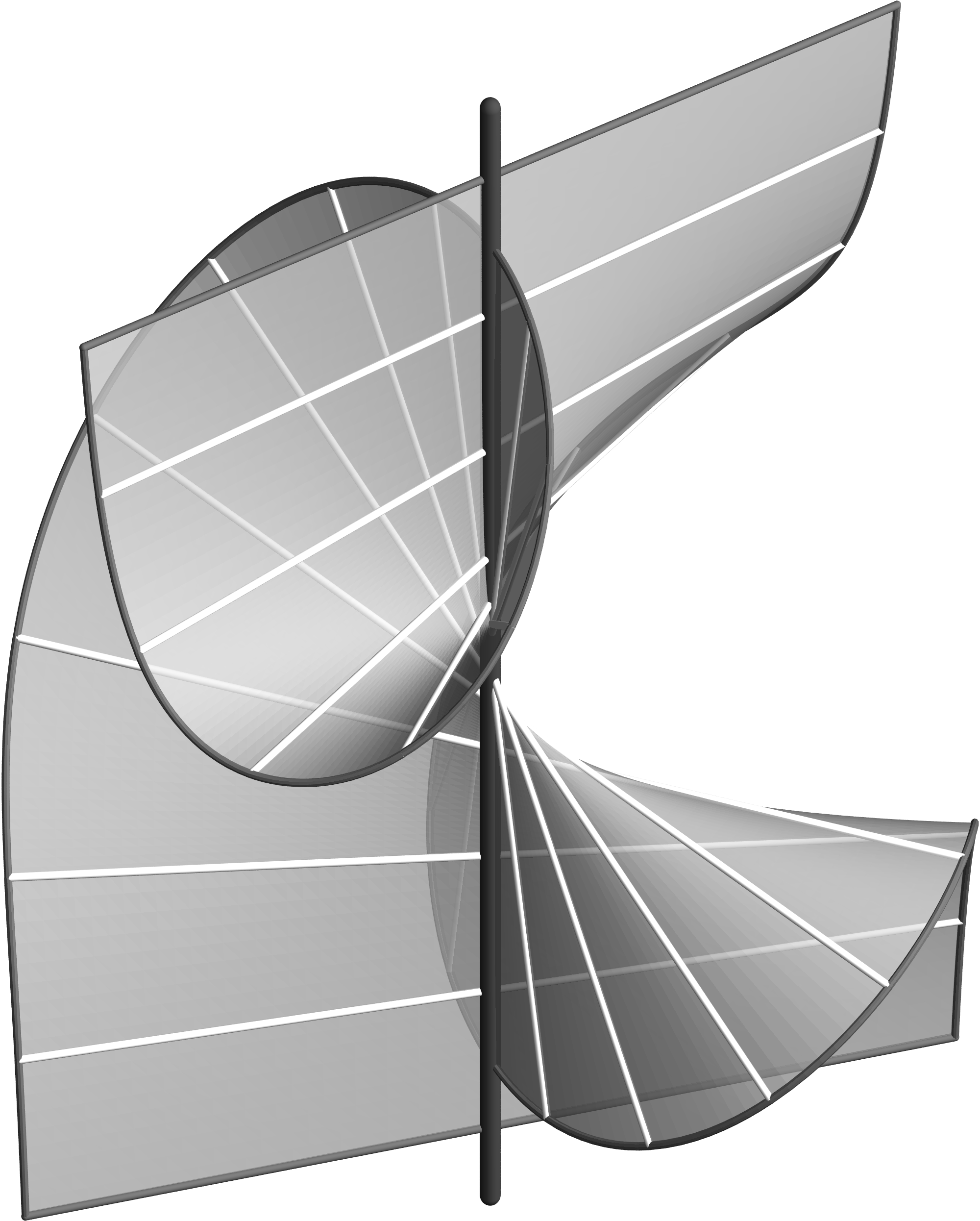}}
         %%\put(0.0 ,0.0){\includegraphics[width=7.0cm]{./bilder/gitter.eps}}
         \put(1.75,4.5){$t$}
         \put(2.1,2.1){$U$}
         \put(3.75,4.5){$F$}
      \end{picture}
      {\refstepcounter{abbildung}\label{abb1}
       \centerline{Figure \ref{abb1}.}}
      \end{minipage}
      \hspace{2.5cm}
      %% Figur 2000 2589 = 4.0 x 5.18 fein
      \begin{minipage}[t]{4.0cm}
         \begin{picture}(4.0,5.18)
         \put(0.0 ,0.0){\includegraphics[width=4.0cm]{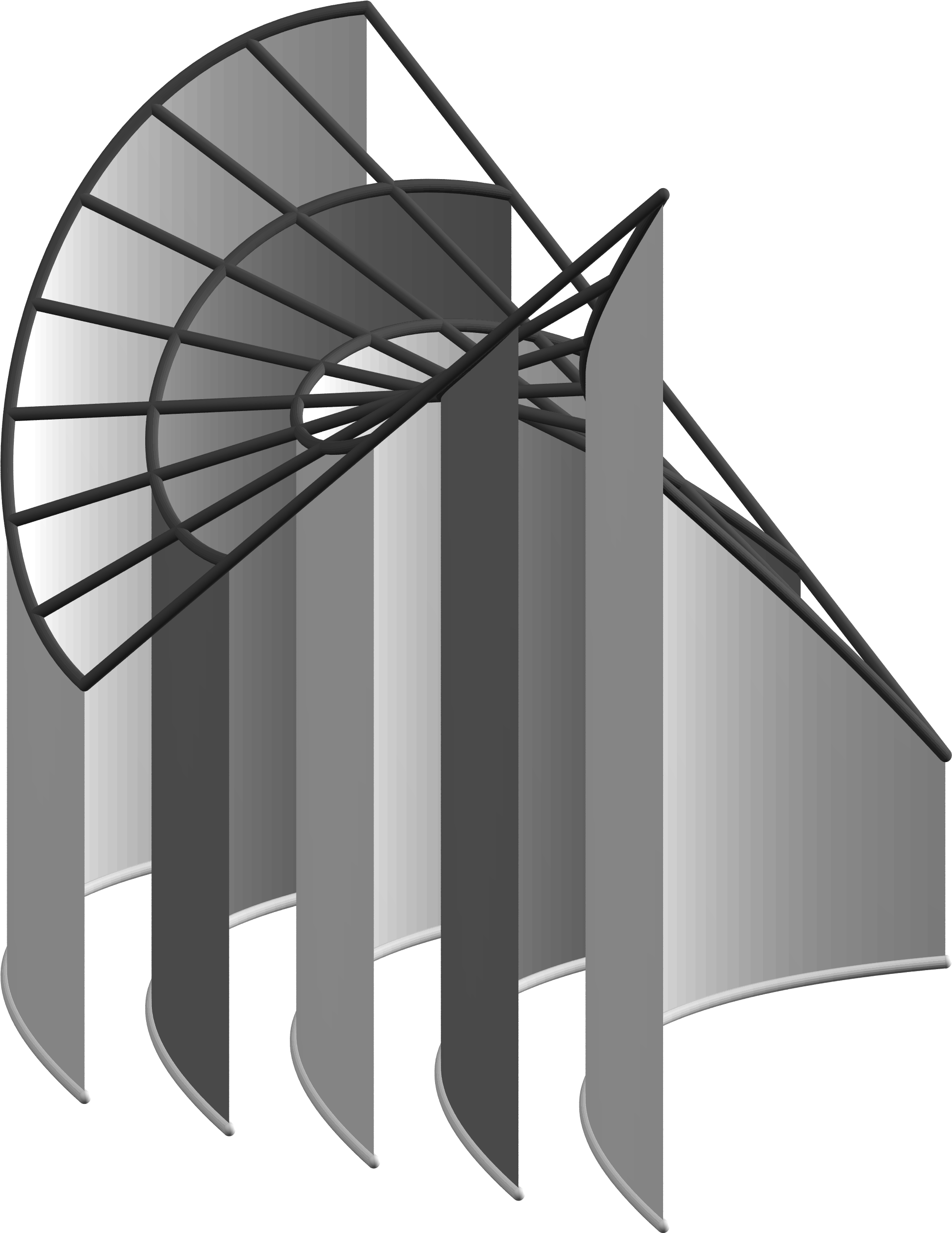}}
      \end{picture}
      {\refstepcounter{abbildung}\label{abb2}
       \centerline{Figure \ref{abb2}.}}
      \end{minipage}
      \end{center}
}%
%%%%%%%%%%%%%%%%%%%%%%%%%%%%%%%%%%%%%%%%%%%%%%%%%%%%%%%%%%%%%%%%%%

\par\para
Our first goal is to describe the order of contact at $U$ of cubic parabolas
given by (\ref{eq:abc_proj}). Since twisted cubics with contact of order five
are identical \cite[pp.~147--148]{bol-50}, we may assume without loss of
generality that the curves are distinct, and that the order of contact is less
or equal four.

\begin{thm}\label{thm:1}
Distinct cubic parabolas $c_\abc$ and $c_\abcq$ on Cayley's ruled
surface have

\vspace{-10pt plus 1pt minus 1pt}%%LAYOUT
\begin{enumerate}\itemsep0pt %%LAYOUT

\item second order contact at $U$ if, and only if, $\beta=\betaq$
or $\beta=3-\betaq$;

\item third order contact at $U$ if, and only if, $\beta=\betaq$
and $\gamma=\gammaq$, or  $\beta=\betaq=\frac32$;

\item fourth order contact at $U$ if, and only if,
$\beta=\betaq=\frac32$ and $\gamma=\gammaq$.
\end{enumerate}
\end{thm}

\begin{proof}
We proceed in two steps:
\par
(i) First, we consider the quadratic forms
\begin{equation*}%%\label{eq:1}
 Q_1:\RR^{4\times 1}\to\RR : \vx \mapsto 6x_0x_3-2x_1x_2,
 \;\;
 Q_2:\RR^{4\times 1}\to\RR : \vx \mapsto 4x_2^2-6x_1x_3
\end{equation*}
which determine a hyperbolic paraboloid and a quadratic cone,
respectively. Their intersection is the cubic parabola
$c_{0,2,0}$, given by
\begin{equation*}
\RR(u_0,u_1)^\T\mapsto
  \RR\left(u_0^3, u_0^2u_1,
      \frac{u_0u_1^2}{2},
   \frac{u_1^3}{6}\right)^\T,
\end{equation*}
and the line $x_2=x_3=0$. The tangent planes of the two surfaces
at $U$ are different.
\par
Next, let $G:=(g_{ij})_{0\leq i,j\leq 3}\in\GL_4(\RR)$ be a lower triangular
matrix, i.e., $g_{ij}=0$ for all $j>i$. The collineation which is induced by
such a matrix $G$ fixes the point $U$, the line $t$, and the plane $\omega$; it
takes $c_{0,2,0}$ to a cubic parabola, say $c'$. In order to determine the
order of contact of $c_{0,2,0}$ and $c'$ we follow \cite[p.~147]{bol-50}. As
$U=\RR\left(\Phi_{0,2,0}((0,1)^\T)\right)$, so we expand for $n=1,2$ the
functions\footnote{Observe that sometimes we do not distinguish between a
linear mapping and its canonical matrix.}
\begin{equation}\label{eq:expand}
    H_n:\RR\to \RR : u_0\mapsto \left(Q_n\circ
    G\circ\Phi_{0,2,0}\right)\left((u_0,1)^\T\right)
    =:\sum_{m=0}^6h_{nm}u_0^m
\end{equation}
in terms of powers of $u_0$ and obtain
\begin{equation}\label{eq:koeffizienten}
\begin{array}{l@{\quad\;}l}
  h_{10} = h_{11} = h_{12}=0, &
  h_{13} = g_{00}g_{33}-g_{11}g_{22},\\
  h_{14} = 3g_{00}g_{32}-g_{10}g_{22}-2g_{11}g_{21},&
  h_{20} = h_{21}=0, \\
  h_{22} = g_{22}^2-g_{11}g_{33}, &
  h_{23} = -g_{10}g_{33} -3g_{11}g_{32}  +4g_{21}g_{22}, \\
  h_{24} = -6g_{11}g_{31}-3g_{10}g_{32}\\
  \hphantom{h_{24} = {}}\;\;+4g_{20}g_{22}+4g_{21}^2; & %%LAYOUT
\end{array}
\end{equation}
the remaining coefficients $h_{15},h_{16},h_{25},h_{26}$ will not
be needed. Note that the matrix entry $g_{30}$ does not appear in
(\ref{eq:koeffizienten}).
\par
(ii) We consider the collineation of $\PP_3(\RR)$ which is induced
by the regular matrix
\begin{equation*}%%\label{eq:M}
M_\abc:=\frac{1}{18\beta(\beta-3)}\,
      \Mat4{  3\beta        & 0         & 0& 0\\
             -3\beta \gamma & 3\beta    & 0& 0\\
              3\alpha       & 0         & 6& 0\\
      \gamma(-3\alpha+\beta\gamma^2)&3(\alpha-\beta\gamma^2)&6\gamma(\beta-1)&-6(\beta-3)
           },
\end{equation*}
where $(\abc)\in\RR^3$ and $\beta\neq 0,3$. Obviously, it fixes
the point $U$ and takes $c_{0,2,0}$ to $c_\abc$, since
\begin{equation*}
  \Phi_\abc=6(\beta-3)M_\abc\circ\Phi_{0,2,0}.
\end{equation*}
The (irrelevant) scalar factor in the definition of $M_\abc$
enables us to avoid fractions in the matrix
\begin{equation*}
M_\abc^{-1} =
      \Mat4{6(\beta-3)       &0&0&0\\
            6\gamma(\beta-3) &6(\beta-3)    &0&0\\
           -3\alpha(\beta-3)    &0&3\beta(\beta-3) &0\\
            \gamma(3\alpha-3\alpha\beta-2\beta\gamma^{2})&3(\alpha-\beta\gamma^{2})&3\beta\gamma(\beta-1)&-3\beta}.
\end{equation*}
The order of contact at $U$ of the cubic parabolas $c_\abc$ and
$c_\abcq$ coincides with the order of contact at $U$ of
$c_{0,2,0}$ and that cubic parabola which arises from $c_{0,2,0}$
under the action of the matrix
%%%%% manuelle Abstandsermittlung fuer Matrix LAYOUT! %%%%%%%%%%%%%%%%
\renewcommand{\abst}[1]{\hspace{0.4mm}{#1}\hspace{0.4mm}}
%%%%%%%%%%%%%%%%%%%%%%%%%%%%%%%%%%%%%%%%%%%%%%%%%%%%%%%%%%%%%%%%%%%%%%
\begin{equation*}
\begin{array}{l}%%LAYOUT
2\betaq(\betaq-3) \, M_\abc^{-1} \cdot M_\abcq = {}\hfill
\\
\schieb{1}
    =\schieb{-1}
    \left(\schieb{-2}\begin{array}{*{3}{c@{\schieb2}}c}%% LAYOUT
    2\betaq ( \beta-3 ) &0&0&0
    \\
    2\betaq ( \beta\fuell{-} 3 )  (\gamma\fuell{-}\gammaq ) & 2\betaq ( \beta-3 ) &0&0
    \\
    (\beta\abst{-}3) (\alphaq\beta \abst{-}\alpha\betaq ) &0&2\beta ( \beta-3 ) &0
    \\
    * &\betaq(\alpha\abst{-}\beta\gamma^2) \abst{-} \beta(\alphaq\abst{-}\betaq\,\gammaq^2)
      &2\beta (\beta\gamma \abst{-} \betaq\,\gammaq \abst{-} \gamma \abst{+} \gammaq)
      &2\beta (\betaq\abst{-}3)
          \end{array}\schieb{-2}\right)\schieb{-0.5}.
\end{array}
\end{equation*}
This matrix takes over the role of the matrix $G$ from the first
part of the proof. (Its entry in the south-west corner has a
rather complicated form and will not be needed). Therefore
$c_\abc$ and $c_\abcq$ have contact of order $k$ at $U$ if, and
only if, in (\ref{eq:expand}) the coefficients $h_{n0}$, $h_{n1}$,
\ldots $h_{nk}$ vanish for $n=1,2$.
\par
By (\ref{eq:koeffizienten}), this leads for $k=2$ to the single
condition
\begin{equation*}
  h_{22}=4\beta(\beta-3)(3-\beta-\betaq)(\betaq-\beta)=0
\end{equation*}
which proves the assertion in (a). By virtue of (a), for $k=3$
there are two cases. If $\beta=\betaq$ then $h_{13}$ vanishes and
we obtain the condition
\begin{equation*}
   h_{23} = 8 \beta^{2} ( \beta-3 )  (2 \beta-3 ) ( \gammaq-\gamma )=0,
\end{equation*}
whereas $\beta=3-\betaq$ yields
\begin{equation*}
  h_{13}  =  4\beta  ( \beta-3 ) ^{2} ( 2 \beta-3 )=0,
  \quad
  h_{23}  =  4\beta  ( \beta-3 ) ^{2} ( 2 \beta-3 ) ( \gamma+2 \gammaq )=0.
\end{equation*}
Altogether this proves (b). Finally, for $k=4$ there again are two
possibilities: If $\beta=\betaq$ and $\gamma=\gammaq$ then
$h_{14}$ vanishes, whence we get
\begin{equation*}
  h_{24}  =  4 \beta^{2} ( \beta-3 )  ( 2 \beta -3 )
 ( \alphaq-\alpha )=0.
\end{equation*}
Note that here $\alpha\neq\alphaq$, since $c_\abc\neq c_\abcq$. On
the other hand, if $\beta=\betaq=\frac32$ then the conditions read
\begin{equation*}
  h_{14}  =  {\frac {81}{2}}\,( \gammaq- \gamma)=0,
  \quad
  h_{24}  =     {\frac {81}{2}}  ( 2 \gammaq+\gamma )  ( \gammaq-\gamma)=0.
\end{equation*}
This completes the proof.
\end{proof}
Alternatively, the preceding results could be derived from
\cite[Theorem~1]{degen-88} which describes contact of higher order between
curves in $d$-dimensional real projective space.
\par\para
In the following pictures we adopt once more the same alternative
point of view like in Figure \ref{abb1}, i.e., the plane with
equation $x_3=0$ is at infinity.
\par
In Figure \ref{abb3} two curves $c_\abc$ and $c_\abcq$ are
displayed. As $(\abc)=(0,\frac{1}{10},0)$ and $(\abcq) =
(1,3-\frac{1}{10},\frac{1}{10})$, they have contact of second
order at $U$.
\par
A family of curves $c_{\alpha,\beta,0}$ with
$\alpha=-3,-2,\ldots,3$ and $\beta=\frac32$ is shown in Figure
\ref{abb4}. All of them have mutually contact of order four at
$U$. These curves are, with respect to the chosen affine chart
($x_3\neq 0$), cubic hyperbolas for $\alpha<0$, a cubic parabola
for $\alpha=0$, and cubic ellipses for $\alpha>0$; the
corresponding values of $\alpha$ are written next to the images of
the curves. See also Figure \ref{abb2} for another picture of this
family, although with different values for $\alpha$ and $x_0=0$ as
plane at infinity.
%%%%%%%%%%%%%%%%%%%%%%%%%%%%%%%%%%%%%%%%%%%%%%%%%%%%%%%%%%%%%%%%%%
{\unitlength1.0cm
      \begin{center}
      %% Figur 2000 2540 = 4.0 x 5.08 fein
      \begin{minipage}[t]{4.0cm}
         \begin{picture}(4.0,5.59)
         \put(0.0,0.25){\includegraphics[width=4.0cm]{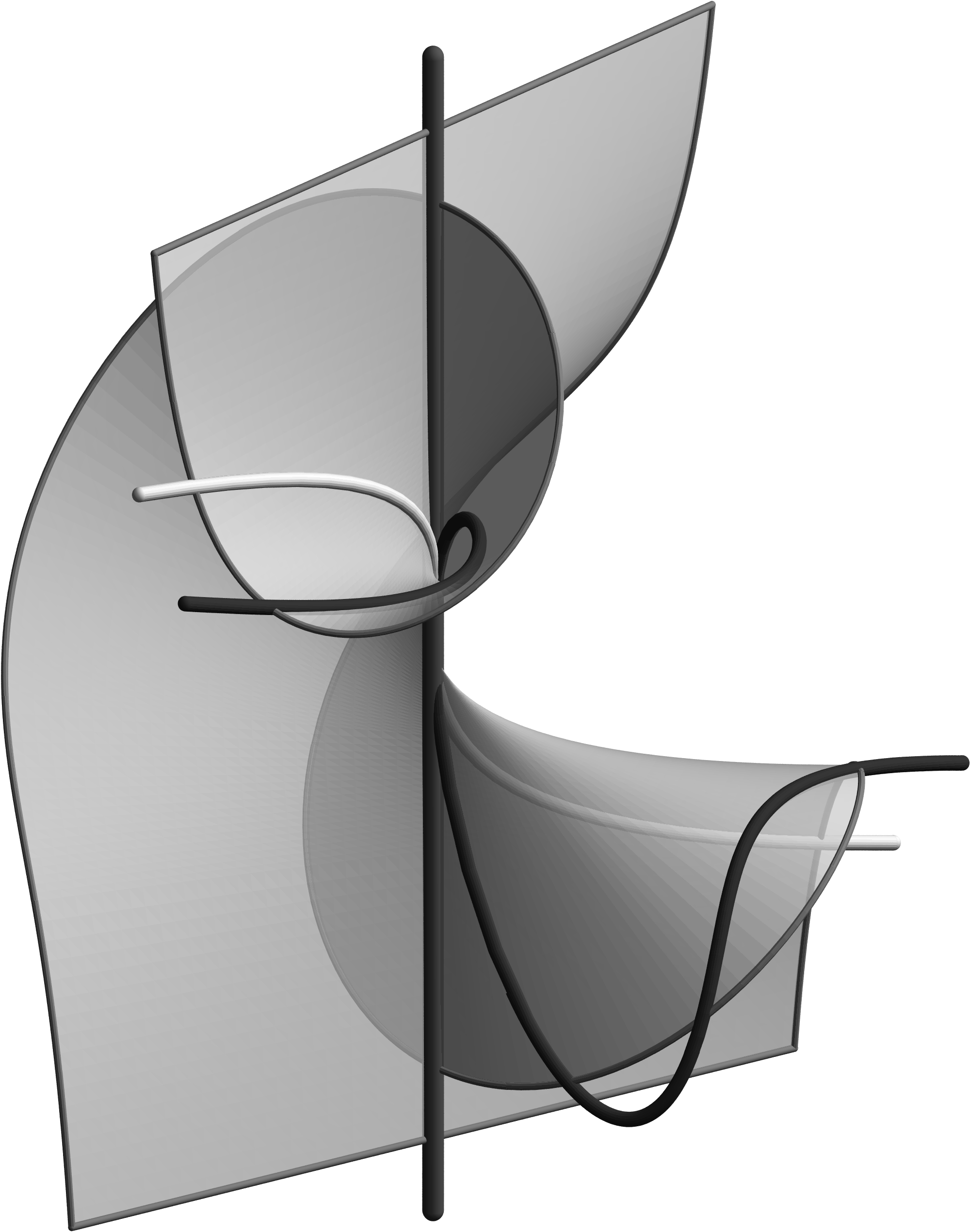}}
       %%\put(0.0 ,0.0){\includegraphics[width=7.0cm]{./bilder/gitter.eps}}
         \put(1.5,4.9){$t$}
         \put(1.9,2.5){$U$}
         \put(3.0,5.0){$F$}
         \put(0.8,3.55){$c_\abc$}
         \put(3.05,2.45){$c_\abcq$}
      \end{picture}
      \vspace{0.5mm}%LAYOUT extra Abstand
      {\refstepcounter{abbildung}\label{abb3}
       \centerline{Figure \ref{abb3}.}}
      \end{minipage}
      \hspace{2.5cm}
      %% Figur 2000 2797 = 4.0 x 5.59 fein
      \begin{minipage}[t]{4.0cm}
         \begin{picture}(4.0,5.59)
         \put(0.0 ,0.0){\includegraphics[width=4.0cm]{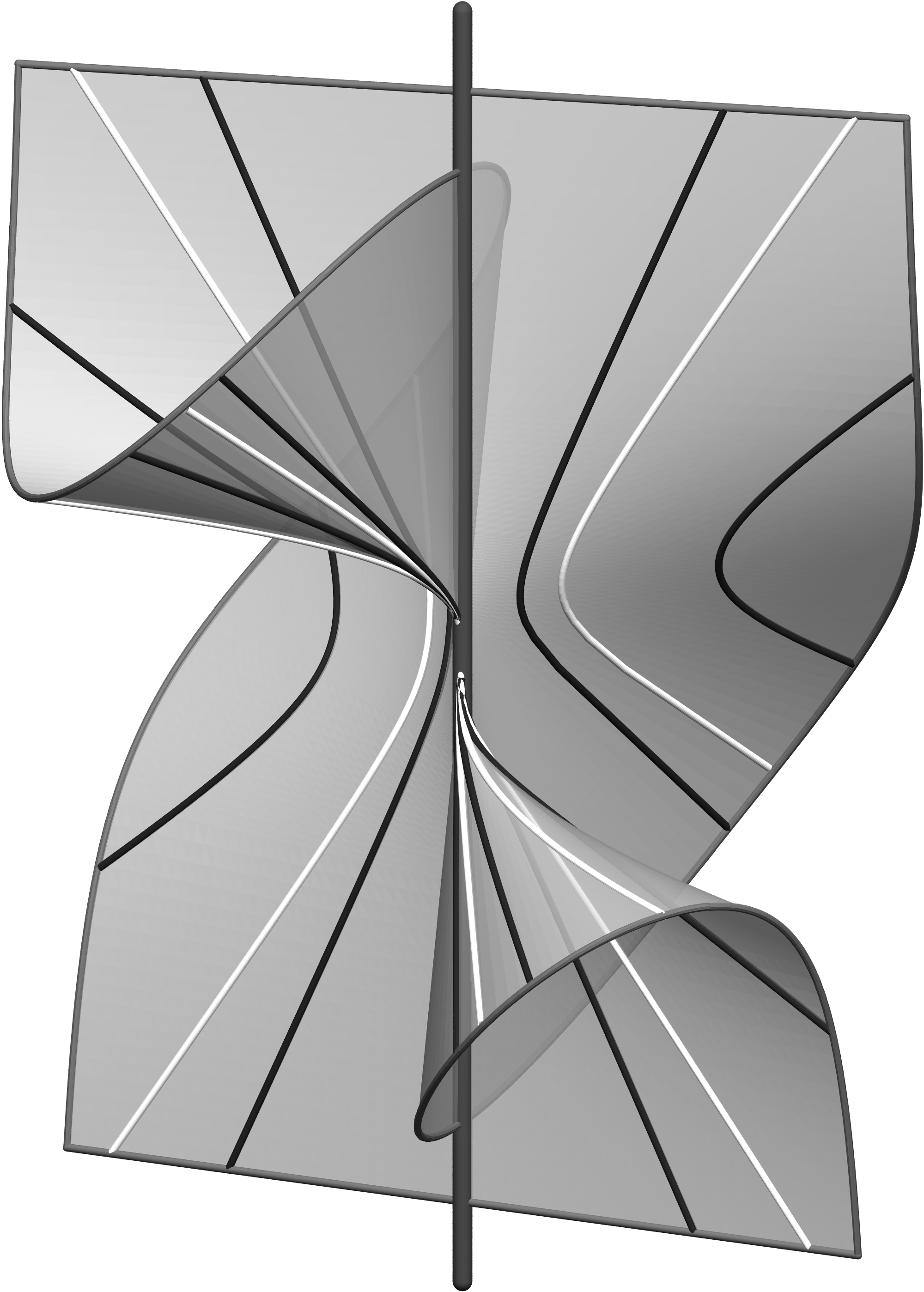}}
         %\put(0.0 ,0.0){\includegraphics[width=7.0cm]{./bilder/gitter.eps}}
         \put(2.15,5.35){$t$}
         \put(4.0,4.6){$F$}
         \put(0.2,3.05){$\scriptstyle 0$}
         \put(2.8,1.35){$\scriptstyle 0$}
         \put(-0.2,4.2){$\scriptstyle 1$}
         \put(3.7,1.0){$\scriptstyle 1$}
         \put(3.5,-.15){$\scriptstyle 2$}
         \put(0.2,5.44){$\scriptstyle 2$}
         \put(0.8,5.40){$\scriptstyle 3$}
         \put(2.9,-.11){$\scriptstyle 3$}
         \put(4.05,3.9){$\scriptstyle -1$}
         \put(3.75,2.6){$\scriptstyle -1$}
         \put(-0.05,1.68){$\scriptstyle -1$}
         \put(2.25,0.5){$\scriptstyle -2$}
         \put(2.4,0.7){\vector(-3,4){0.3}}
         \put(1.5,0.7){\vector(3,4){0.4}}
         \put(1.35,0.55){$\scriptstyle -3$}
         \put(3.5 ,5.2){$\scriptstyle -2$}
         \put(3.42,2.15){$\scriptstyle -2$}
         \put(2.15,1.1){$\scriptstyle -1$}
         \put(2.9 ,5.24){$\scriptstyle -3$}
         \put(3.2,1.8){$\scriptstyle -3$}
         \put(0.11,0.22){$\scriptstyle -2$}
         \put(0.71 ,0.18){$\scriptstyle -3$}
         \put(0.1,3.6){$\scriptstyle -1$}
         \put(0.365,4.28){\vector(1,-1){0.4}}
         \put(0.065,4.34){$\scriptstyle -2$}
         \put(0.5,4.55){\vector(3,-4){0.4}}
         \put(0.2,4.61){$\scriptstyle -3$}
      \end{picture}
      \vspace{0.5mm}%% LAYOUT extra Abstand
      {\refstepcounter{abbildung}\label{abb4}
       \centerline{Figure \ref{abb4}.}}
      \end{minipage}
      \end{center}
}%
%%%%%%%%%%%%%%%%%%%%%%%%%%%%%%%%%%%%%%%%%%%%%%%%%%%%%%%%%%%%%%%%%%

\par\para\label{para:2.4}
It follows from Theorem \ref{thm:1} that cubic parabolas $c_\abc$
with $\beta=\frac32$ play a special role. In order to explain this
from a geometric point of view we consider the \emph{tangent
surface} of a cubic parabola $c_\abc$ and, in particular, its
intersection with the plane at infinity. It is well known that
this is a conic $p_\abc$ together with the line $t$. In fact, via
the first derivative of the local parametrization
$\RR\to\PP_3(\RR) : u_1\mapsto \RR\big(\Phi_\abc((1,u_1)^\T)\big)$
of $c_\abc$ we see that $p_\abc\setminus\{U\}$ is given by
\begin{equation}\label{eq:p_lokal}
  u_1 \mapsto
  \RR\left(0,
        1,
        \frac {2u_1}{\beta},
        \frac{3-\beta}{\beta}\,u_1^2
          +\frac{2\gamma (\beta-1)}{\beta}\, u_1
          +\frac{\alpha}{\beta}-\gamma^2\right)^\T.
\end{equation}
The plane at infinity carries in a natural way the structure of an
\emph{isotropic\/} (or \emph{Galileian\/}) \emph{plane\/} with the absolute
flag $(U,t)$. Each point $\RR(0,1,x_1,x_2)^\T \in \omega\setminus t$ can be
identified with the point $(x_1,x_2)^\T\in\RR^{2\times 1}$. In this way the
standard basis of $\RR^{2\times 1}$ determines a unit length and a unit angle
in the isotropic plane \cite[pp.~11--16]{sachs-87}.
\par
From this point of view each $p_\abc$ is an \emph{isotropic circle}. By
(\ref{eq:p_lokal}), its \emph{isotropic curvature} \cite[p.~112]{sachs-87}
equals $\frac1 2\,\beta(3-\beta)\leq\frac98$; this bound is attained for
$\beta=\frac32$.
\par
It is well known that two isotropic circles $p_\abc$ and $p_\abcq$ have second
order contact at the point $U$ if, and only if, their isotropic curvatures are
the same \cite[pp.~41--42]{sachs-87}, i.e.\ for $\beta=\betaq$ or for
$\beta=3-\betaq$. From this observation one could also derive the assertion in
Theorem \ref{thm:1} (a) as follows: We introduce an auxiliary euclidean metric
in a neighbourhood of $U$, and we take into account that the ratio of the
euclidean curvatures at $U$ of the curves $c_\abc$ and $p_\abc$ (the curves
$c_\abcq$ and $p_\abcq$) equals $4:3$; see \cite[p.~212]{strub-64} for this
theorem of E.~Beltrami.
\par
The flag $(U,t,\omega)$ turns $\PP_3(\RR)$ into a \emph{twofold
isotropic} (or \emph{flag}) \emph{space}. The definition of metric
notions in this space is based upon the identification of
$\RR(1,x_1,x_2,x_3)^\T\in\PP_3(\RR)\setminus\omega$ with
$(x_1,x_2,x_3)^\T\in\RR^{3\times 1}$, and the canonical basis of
$\RR^{3\times 1}$; see \cite{brau-66}.
\par
By \cite[p.~137]{brau-67a}, each cubic parabola $c_\abc$ has the \emph{twofold
isotropic conical curvature} $\frac12\,\beta (3-\beta)\leq\frac98$. Hence the
following characterization follows.
\begin{thm}
Among all cubic parabolas $c_\abc$ on the Cayley surface $F$, the
cubic parabolas with $\beta=\frac32$ are precisely those with
maximal twofold isotropic conical curvature.
\end{thm}
\par
Yet another interpretation is as follows: The regular matrix
\begin{equation*}
  B_\beta:=\diag\left(1,\frac{3-\beta}\beta\,,\frac{3-\beta}\beta\,,\frac{3-\beta}\beta\right),
  \mbox{ where }\beta\in\RR\setminus \{ 0,3 \},
\end{equation*}
yields a \emph{homothetic transformation} of $\PP_3(\RR)$ which
maps the cubic parabola $c_{0,\beta,0}$ to $c_{0,3-\beta,0}$,
since
\begin{equation*}
  (B_\beta\circ\Phi_{0,\beta,0})\left( (u_0,u_1)^\T)\right)
  =
  \Phi_{0,3-\beta,0}\Big(\Big(u_0,\frac{3-\beta}\beta\,u_1\Big)^\T\Big)
  \mbox{ for all } (u_0,u_1)^\T\in\RR^{2\times 1}.
\end{equation*}
As all points at infinity are invariant, the corresponding
isotropic circles $p_{0,\beta,0}$ and $p_{0,3-\beta,0}$ coincide.
This homothetic transformation is identical if, and only if,
$\beta=\frac32$.
\par
The Cayley surface $F$ admits a $3$-parameter collineation group; see
\cite[p.~96]{brau-64} formula (9). The action of this group on the family of
all cubic parabolas $c_\abc$ is described in \cite[p.~97]{brau-64}, formula
(12). (In the last part of that formula some signs have been misprinted. The
text there should read $\alphaq =
-a_0^2\frac{\beta^2}4-a_0a_1\beta\gamma+a_1^2\alpha+b_0\beta$). By virtue of
this action, our previous result on homothetic transformations can be
generalized to other cubic parabolas on $F$.

\section{Dual contact of higher order}\label{se:dualkontakt}

\para
The question remains how to distinguish between cubic parabolas
$c_\abc$ and $c_\abcq$ satisfying the first condition
($\beta=\betaq$) in Theorem \ref{thm:1} (a), and those which meet
the second condition ($\beta=3-\betaq$). A similar question arises
for the two conditions in Theorem \ref{thm:1} (b). We shall see
that such a distinction is possible if we consider the \emph{dual
curves} which are formed by the osculating planes (i.e.\ cubic
developables). Recall that $c_\abc$ and $c_\abcq$ have, by
definition, \emph{dual contact of order $k$} at a common
osculating plane $\sigma$, if their dual curves have contact of
order $k$ at the ``point'' $\sigma$ of the dual projective space.
\par
We shall identify the dual of $\RR^{4\times 1}$ with the vector
space $\RR^{1\times 4}$ in the usual way; so planes (i.e.\ points
of the dual projective space) are given by non-zero \emph{row
vectors}. Thus, for example, a plane $\RR(y_0,y_1,y_2,y_3)$ is
tangent to the Cayley surface (\ref{eq:cayley}) if, and only if,
\begin{equation}\label{eq:dual_cayley}
   3y_0y_3^2 - 3y_1y_2y_3 + y_2^3=0.
\end{equation}
We note that all these tangent planes comprise a Cayley surface in
the dual space.

For each twisted cubic there exists a unique null polarity
(symplectic polarity) which takes each point of the twisted cubic
to its osculating plane. In particular, the null polarity of the
cubic parabola $c_{0,2,0}$ is induced by the linear bijection
\begin{equation}\label{eq:null}
    \RR^{4\times 1} \to \RR^{1\times 4}: \vx\mapsto
    (N_{0,2,0}\cdot\vx)^\T
    \mbox{ with } N_{0,2,0}:=\RMat4{0&0&0&1\\0&0&-1&0\\0&1&0&0\\-1&0&0&0}.
\end{equation}
\par
We are now in a position to prove the following result.

\begin{thm}\label{thm:2}
Distinct cubic parabolas $c_\abc$ and $c_\abcq$ on Cayley's ruled
surface have

\vspace{-10pt plus 1pt minus 1pt}%% LAYOUT
\begin{enumerate}\itemsep0mm %% LAYOUT

\item second order dual contact at $\omega$ if, and only if,
$\beta=\betaq$;

\item third order dual contact at $\omega$ if, and only if,
$\beta=\betaq$ and $\gamma=\gammaq$, or  $\beta=\betaq=\frac52$;

\item fourth order dual contact at $\omega$ if, and only if,
$\beta=\betaq=\frac73$ and $\gamma=\gammaq$.
\end{enumerate}
\end{thm}

\begin{proof}
The matrix $(M_\abc^\T)^{-1}\cdot N_{0,2,0}$ determines a duality
of $\PP_3(\RR)$ which maps the set of points of $c_{0,2,0}$ onto
the set of osculating planes of $c_\abc$. Since the product of a
duality and the inverse of a duality is a collineation, we obtain
the following:
\par
The order of dual contact at $\omega$ of the given curves $c_\abc$
and $c_\abcq$ coincides with the order of contact at $U$ of the
cubic parabola $c_{0,2,0}$ and that cubic parabola which arises
from $c_{0,2,0}$ under the collineation given by the matrix
\begin{equation*}
\begin{array}{l}
  2\beta(\beta-3)\, N_{0,2,0}^{-1}\cdot
  M_\abc^{\T}\cdot(M_\abcq^\T)^{-1}\cdot N_{0,2,0} =\hfill
   \\
   \schieb{1.0}=\schieb{-0.5}
   \left(\schieb{-1}\begin{array}{*{3}{c@{\schieb2}}c}%%LAYOUT Abstand kleiner
   2( \beta-3 ) \betaq&0&0&0   \\
   2  \betaq ( \beta\gamma - \betaq\,\gammaq - \gamma + \gammaq  ) &
    2\betaq( \betaq-3 ) { }&0&0 \\
     \alphaq\beta - \alpha\betaq + \beta\betaq(\gamma^{2} - \gammaq^{2} )&
    0&
    2\beta (  \betaq - 3 ) &0\\
    * &
     ( \betaq - 3 )  ( \alpha\betaq  - \alphaq\beta ) &
     2\beta (  \betaq - 3 ) ( \gamma - \gammaq )&
     2\beta (  \betaq  - 3 )
     \end{array}\schieb{-1}\right)\schieb{-0.5}.
\end{array}
\end{equation*}
Here $*$ denotes an entry that will not be needed.
\par
We now proceed as in the proof of Theorem \ref{thm:1}. By
substituting the entries of the matrix above into
(\ref{eq:koeffizienten}), we read off necessary and sufficient
conditions for dual contact of order $k$ at the plane $\omega$ of
$c_\abc$ and $c_\abcq$.
\par
For $k=2$ we get the single condition
\begin{equation*}
  h_{22}= 4 \beta ( \betaq-3 ) ^{2} ( \beta-\betaq ) = 0
\end{equation*}
which proves the assertion in (a). By (a), we let $\beta=\betaq$
for the discussion of $k=3$. Then $h_{13}$ vanishes and we arrive
at the condition
\begin{equation*}
   h_{23} = 8\beta^{2}( \beta-3 )  ( 2\beta-5 )
 ( \gammaq-\gamma ) =0,
\end{equation*}
from which (b) is immediate. Finally, for $k=4$ we distinguish two
cases: If $\beta=\betaq$ and $\gamma=\gammaq$ then $h_{14}$
vanishes and we are lead to the condition
\begin{equation*}
  h_{24}  = 4 \beta^{2}( \beta-3 )  ( 3\beta-7 )
 ( \alphaq-\alpha ) =0.
\end{equation*}
Note that here $\alpha\neq\alphaq$, since $c_\abc\neq c_\abcq$.
The proof of (c) will be finished by showing that the case
$\beta=\betaq=\frac52$ does not occur. From the assumption
$\beta=\betaq=\frac52$ follows the first condition
\begin{equation*}
  h_{14}  = {\frac {75}{2}}\,(\gamma-\gammaq) =0.
\end{equation*}
Now, letting $\gamma=\gammaq$, the second condition
\begin{equation*}
  h_{24}  = {\frac {25}{4}}\,(\alpha-\alphaq) =0
\end{equation*}
is obtained. However, both conditions cannot be satisfied
simultaneously, since the first condition and $c_\abc\neq c_\abcq$
together imply that $\alpha\neq\alphaq$.
\end{proof}
\par\para
By combining the results of Theorem \ref{thm:1} and Theorem
\ref{thm:2}, it is an immediate task to decide whether or not two
(not necessarily distinct) cubic parabolas $c_\abc$ and $c_\abcq$
have contact at $U$ and at the same time dual contact at $\omega$
of prescribed orders. In particular, we infer that two cubic
parabolas of this kind, with fourth order contact at $U$ and
fourth order dual contact at $\omega$, are identical.

\par\para\label{para:3.2}
In this section we aim at explaining how the results of Theorems
\ref{thm:1} and \ref{thm:2} are related to each other.
\par
Let us choose a \emph{fixed} real number $\beta\neq 0,3$. We
consider the local parametrization
\begin{equation*}
    \Psi_\beta:\RR^2\to \PP_3(\RR) :
    (\alpha,u)\mapsto\RR\left(\Phi_{\alpha,\beta,0}((1,u)^\T)\right)
\end{equation*}
of $F$; its image is $F\setminus t$, i.e. the affine part of $F$.
For our fixed $\beta$ and $\gamma=0$ the affine parts of the
parabolic cylinders (\ref{eq:zylinder}) form a partition of
$\PP_3(\RR)\setminus\omega$; see Figure \ref{abb2}. Hence
$\Psi_\beta$ is injective so that through each point $P\in
F\setminus t$ there passes a unique curve $c_{\alpha,\beta,0}$.
Consequently, we can define a mapping $\Sigma$ of $F\setminus t$
into the dual projective space by
\begin{equation}\label{eq:Sigma}
  P\in c_{\alpha,\beta,0}\setminus\{U\}\stackrel\Sigma\longmapsto
  \mbox{ osculating plane of } c_{\alpha,\beta,0} \mbox{ at } P.
\end{equation}
\begin{thm}
The image of the affine part of the Cayley surface $F$ under the mapping
$\Sigma$ described in \emph{(\ref{eq:Sigma})} consists of tangent planes of a
Cayley surface for $\beta\neq 0,3,\frac83$, and of tangent planes of a
hyperbolic paraboloid for $\beta=\frac83$.
\end{thm}
\begin{proof}
As the null polarity of $c_{\alpha,\beta,0}$ arises from the
matrix
\begin{eqnarray}\label{eq:N_allgemein}
    N_{\alpha,\beta,0}&:=&
    (M_{\alpha,\beta,0}^{-1})^\T \cdot N_{0,2,0} \cdot
    M_{\alpha,\beta,0}^{-1}  \nonumber \\
    &{}=& %% diese leere Klammer ist wichtig!
    18(\beta-3)\,
    \Mat4{
    0&- \alpha  ( \beta-4 ) &0&- \beta  \\
    \alpha  ( \beta-4 ) &0&-\beta( \beta-3 )&0\\
    0&\beta( \beta-3 ) &0&0\\
    \beta&0&0&0
    },
\end{eqnarray}
so the $\Sigma$-image of a point
$P=\RR\left(\Phi_{\alpha,\beta,0}((1,u)^\T)\right)$ is the plane
which is described by the non-zero row vector
\begin{equation}\label{eq:polar}
  6(\beta-3) \Big(
  ( \beta-3)({u}^{2}- 3\alpha) u,
  -3(\beta -3){u}^{2} -3\alpha ,
  3 \beta ( \beta-3 ) u,
  3\beta
  \Big).
\end{equation}
In discussing $\Sigma(F\setminus t)$ there are two cases:

(i) Suppose that $\beta\neq\frac83$. Then a duality of
$\PP_3(\RR)$ is determined by the regular matrix
\begin{equation*}
    D_\beta:=
    \frac {18}{  \beta-3 }
    \Mat4{
    0&0&0&-(3\beta-8)\\
    0&0&-(3\beta-8)&0\\
    0&\beta( \beta-3 )^{2}&0&0\\
    \beta( \beta-3 )^{2}&0&0&0
    }.
\end{equation*}
Letting
\begin{equation}\label{eq:betastrich}
  \alpha':=\alpha(\beta-3) \mbox{ and }\beta':= {\frac
  {3\beta-8}{\beta-3}},
\end{equation}
the transpose of
$(D_\beta\circ\Phi_{\alpha',\beta',0})\left((1,(\beta-3)u)^\T\right)$
is easily seen to equal the row vector in (\ref{eq:polar}). Hence
$\Sigma(F\setminus t)$ is part of a Cayley surface in the dual
space which in turn, by (\ref{eq:dual_cayley}), is the set of
tangent planes of a Cayley surface in $\PP_3(\RR)$.
\par
(ii) If $\beta =\frac83$ then the row vector (\ref{eq:polar})
simplifies to
\begin{equation*}
  -2\left(\frac{  -( u^2-3\alpha )u}{3} ,
         u^2-3\alpha,-\frac{8u}{3},
         8\right)
\end{equation*}
Thus the set $\Sigma(F\setminus t)$ is part of the non-degenerate
ruled quadric in the dual space with equation $y_0y_3-y_1y_2=0$
(in terms of dual coordinates). In other words, $\Sigma(F\setminus
t)$ consists of tangent planes of a hyperbolic paraboloid in
$\PP_3(\RR)$.
\end{proof}

Let us add the following remark. The linear fractional
transformation
\begin{equation*}
   \Lambda :\RR\cup\{\infty\}\to\RR\cup\{\infty\} : \xi\mapsto \frac
  {3\xi-8}{\xi-3}
\end{equation*}
is an involution such that our fixed $\beta\neq 0,3,\frac83$ goes
over to $\beta'$, as defined in (\ref{eq:betastrich}), whereas
$\Lambda(\frac83)=0$. In particular, if $\beta=\frac73$ then
$\beta'=\Lambda(\beta)=\frac32$. This explains the relation
between Theorem \ref{thm:1} (c) and Theorem \ref{thm:2} (c).  Also
the fixed values of $\Lambda$ are noteworthy:
\par
For $\beta=\Lambda(\beta)=2$ the curves $c_{\alpha,2,0}$ are
\emph{asymptotic curves} of $F$, i.e., the osculating plane of
$c_{\alpha,2,0}$ at each point $P\neq U$ is the tangent plane of
$F$ at $P$. This means that the planes of the set
$\Sigma(F\setminus t)$ are tangent planes of $F$ rather than
tangent planes of another Cayley surface.
\par
For $\beta=\Lambda(\beta)=4$ it is immediate form
(\ref{eq:N_allgemein}) that the matrix $N_{\alpha,4,0}$ does not
depend on the parameter $\alpha\in\RR$, whence in this particular
case the mapping $\Sigma$ is merely the restriction of a null
polarity of $\PP_3(\RR)$ to the affine part of the Cayley surface
$F$.
\par\para
There remains the problem to find a geometric interpretation of
the value $\beta=\frac52$ which appears in Theorem \ref{thm:2}
(b).

\textbf{Acknowledgement.} The author is grateful to Friedrich
Manhart for many inspiring discussions.

%\bibliographystyle{plain}
%\bibliography{d:/forschung/separata/ketten}

\small%% bis zum Schluss

Hans Havlicek\\ Institut f\"ur Geometrie\\ Technische Universit\"at\\
Wiedner Hauptstra{\ss}e 8--10/1133\\ A-1040 Wien\\ Austria\\
email: \texttt{havlicek@geometrie.tuwien.ac.at}
\end{document}